%% file: arxiv_04_18_25.tex
\documentclass[12pt]{amsart}

\input{macros}

\title{On Invariants of Artin-Schreier Curves}

\author[J. Duque-Rosero]{Juanita Duque-Rosero}
\address{Juanita Duque-Rosero, Department of Mathematics and Statistics, Boston University, 665 Commonwealth Ave, Boston, MA 02215, USA}
\curraddr{}
\email{juanita@bu.edu}
\urladdr{\url{https://juanitaduquer.github.io}}

\author[H. Goodson]{Heidi Goodson}
\address{Heidi Goodson, Department of Mathematics, Brooklyn College; 2900 Bedford Avenue, Brooklyn, NY 11210, USA}
\curraddr{}
\email{heidi.goodson@brooklyn.cuny.edu}
\urladdr{\url{https://sites.google.com/site/heidigoodson}}

\author[E. Lorenzo Garc\'ia]{Elisa Lorenzo Garc\'ia}
\address{Elisa Lorenzo Garc\'ia, Institut de Math\'ematiques, Universit\'e de Neuch\^atel, Rue Emile-Argand 11, 2000, Neuch\^atel, Switzerland -- Laboratoire IRMAR, Office 602,
Universit\'e de Rennes 1, Campus de Beaulieu, 35042, Rennes Cedex}
\curraddr{}
\email{elisa.lorenzo@unine.ch, elisa.lorenzogarcia@univ-rennes1.fr }
\urladdr{\url{https://sites.google.com/site/elisalorenzo}}

\author[B. Malmskog]{Beth Malmskog}
\address{Beth Malmskog, Department of Mathematics and Computer Science, Colorado College, 14 E Cache la Poudre, Colorado Springs, CO 80903, USA}
\curraddr{}
\email{bmalmskog@coloradocollege.edu}
\urladdr{\url{https://malmskog.wordpress.com}}

\author[R. Scheidler]{Renate Scheidler}
\address{Renate Scheidler, Department of Mathematics and Statistics, University of Calgary, 2500 University Drive NW, Calgary, Alberta, Canada T2N 1N4}
\curraddr{}
\email{rscheidl@ucalgary.ca}
\urladdr{\url{https://cpsc.ucalgary.ca/~rscheidl}}

\begin{document}


\begin{abstract} The main goal of this article is to expand the theory of invariants of Artin-Schreier curves by giving a complete classification in genus 3 and 4. To achieve this goal, we first establish standard forms of Artin-Schreier curves and determine all isomorphisms between curves in this form. We then compute reconstructing systems of invariants for curves in each connected component of the strata of the moduli spaces for Artin-Schreier curves of genus 3 and 4 for $p>2$. 
\end{abstract}

\maketitle

\section{Introduction} \label{sec:Intro}
Artin-Schreier theory derives from a 1927 paper of Emil Artin and Otto Schreier characterizing degree~$p$ Galois extensions of a  field of characteristic $p$ \cite{artin1927kennzeichnung}. Artin and Schreier proved that such an extension $K(y)/K$ is precisely the splitting field of a polynomial in $y$ of the form $y^p-y-\alpha$, where $\alpha \in K$ and $\alpha \neq \beta^p-\beta$ for any $\beta \in K$. Such extensions are now referred to as Artin-Schreier extensions, and the varieties defined by the associated polynomials are called Artin-Schreier curves. In this article, we aim to parameterize moduli spaces of Artin-Schreier curves in characteristic $p>2$ by computing invariants of the curves.

It is well-known that elliptic curve isomorphism classes are given by the their $j$-invariants. 
In genus 2, we can describe isomorphism classes of curves using Igusa invariants \cite{Igusa}. 
There are also examples of invariant computations in higher genus. For example, there are results for some genus 3 curves in \cite{Shioda,Dixmier,Ohno}, hyperelliptic curves in \cite{LR12}, Picard curves in \cite{KLS}, Ciani curves in \cite{BCKKLS}, and others. In general, these calculations are done over the complex numbers, but they usually extend well to fields of large characteristic (genus 2 curves in \cite{Liu}, genus 3 curves in \cite{basson}, \cite[Sec.\ 4]{LLLR}). The sources \cite[Sec.\ 14.5]{handbook} and \cite[Sec.\ 7.4.3]{LiuBook} specifically cover Artin-Schreier curves in characteristic 2, which are hyperelliptic. However, there are still many cases missing from the literature, including ``small cases" such as non-hyperelliptic curves of genus $g=3$ in characteristic $p=3$. 

Artin-Schreier curves have undergone intense research not only as objects of mathematical interest in their own right, but also for their applications to coding theory (see \cite{van1991artin} for example). For certain families of Artin-Scheier curves over a rational function field $K = \F_q(x)$, where $q$ is a power of $p$, it is possible to provide point counts, find their automorphism groups, and determine their zeta functions \cite{van1992reed}, \cite{BHMSSV}. Such families frequently contain curves with many points, including maximal curves with respect to the Weil bound. Investigations of Artin-Schreier point counts from the perspective of arithmetic statistics can be found in \cite{entin2012distribution},\cite{bucur2012distribution}, \cite{bucur2016statistics}. 

In this paper, we compute invariants for 
Artin-Schreier curves of genus 3 and 4 in characteristic $p>2$. 
Our work builds on that of Pries and Zhu, who determined the strata and dimensions of irreducible components of the moduli spaces of Artin-Schreier curves \cite{PriesZhu2012}. We seek to characterize these moduli spaces concretely by parameterizing them via invariants. Our main ideas are taken from Geometric Invariant Theory (GIT), which was first developed by Mumford in 1965 \cite{GIT}. This theory provides techniques for forming the ``quotient'' of an algebraic variety (or scheme) $X$ by a group $G$ and is especially useful for constructing moduli spaces as quotients of schemes parameterizing objects. We provide an overview of invariant theory in Section~\ref{sec:Invariants} and of Artin-Schreier curves, their moduli spaces, and their isomorphisms in Section~\ref{sec:AScurves}.

As in the genus 1 and 2 cases, it is useful to work with a standard form for Artin-Schreier curves in order to determine their invariants. Section~\ref{sec:StandardForm} develops a standard form for Artin-Schreier curves. We also describe the isomorphisms between standard forms in this section.
In Section~\ref{sec:DetailedExample}, we determine reconstructing systems of invariants for all Artin-Schreier curves of genus 3 and 4 for primes $p>2$, as listed in Table~\ref{tab:smallgenustable}.  This allows us to prove our main result, given in the following theorem.

\begin{ithm}\label{thm:mainTheorem}
A system of reconstructing invariants for all Artin-Schreier curves of genus $g=3,4$ in characteristic $p>2$ is given in Table~\ref{tab:invariants}.
\end{ithm}

For these curves, the strata and dimensions of the irreducible components of their moduli spaces are completely characterized in Table~\ref{tab:smallgenustable}.  The standard forms (see Theorem \ref{T:normal}) and invariants listed in Table~\ref{tab:invariants} are derived in the examples in Section~\ref{sec:DetailedExample}.

\begin{table}[ht]
\centering
\bgroup
\def\arraystretch{2.4}
\begin{tabular}{l|l|l||l||l}
$\bm{g}$ &$\bm{p}$& $\bm{s}$ & \text{\textbf{Standard form}}&\text{\textbf{Set of Reconstructing invariants over }}$\bm{\Fpbar}$   \\ \hline \hline  
$3$&$3$&$0$&$\displaystyle y^3-y=x^4+ax^2$& $\displaystyle \{a^4\}$\\\hline
$3$&$3$&$2$&$\displaystyle y^3-y=x^2+ax+\frac bx  $&$\displaystyle \{a^4,ab,b^4\}$\\\hline
$3$&$7$&$0$&$\displaystyle y^7-y=x^2$&$\emptyset$\\\hline
$4$&$3$&$0$&$\displaystyle y^3-y=x^5+cx^4+dx^2$&
$\substack{\displaystyle \{(c^3+d)^{10},(-cd-\epsilon^2)^5,(c^3+d)^2(-cd-\epsilon^2)\} \\\displaystyle\text{ where }\epsilon^3=c \medskip}$ \\\hline
$4$&$3$&$2$& $\displaystyle y^3-y=x^2+ax+\frac bx+\frac{c}{x^2}$&$\displaystyle \{c,ab,a^4c^2-b^4\}$\\\hline
$4$&$3$&$4$& $\displaystyle y^3-y=x^2+ax+\frac bx+\frac{c}{x-1}$&$\displaystyle \{(abc)^2,(abc)(a-b-c),ab+ac-bc\}$\\\hline
$4$&$5$&$0$&$\displaystyle y^5-y=x^3+ax^2$&$\displaystyle \{a^{12}\}$\\\hline
$4$&$5$&$1$&$\displaystyle y^5-y=x+\frac ax$&$\displaystyle \{a^2\}$ \medskip
\end{tabular}
\egroup
\caption{Reconstructing invariants for all Artin-Schreier curves of genus $g~=~3,4$ in characteristic $p>2$.  The curves are classified by their $p$-rank $s$.}
\label{tab:invariants}
\end{table}
The remainder of this paper is devoted to the proof of Theorem \ref{thm:mainTheorem}. We start by computing sets of invariants directly.  We then show that these sets allow us to reconstruct 
standard forms of Artin-Schreier curves.  Finally, 
standard invariant theoretic results summarized in Section~\ref{sec:Background} establish that these sets must generate the full invariant ring.

\subsection*{Acknowledgements}
The authors 
thank the Banff International Research Station for Mathematical Innovation and Discovery in Banff (Canada) for sponsoring the Women in Numbers~6 (WIN6) workshop and for providing a productive and enjoyable environment for our initial
work on this project. We would especially like to thank the organizers of WIN6, Shabnam Akhtari,
Alina Bucur, Jennifer Park and Renate Scheidler, for making the conference and this
collaboration possible. We 
are indebted to Rachel Pries for helpful conversations around this work. Last but not least, we wish to thank the anonymous referee for useful comments and suggestions.

The research of the first author is partially supported by the Simons Foundation grant \#550023. The research of the second author is in part supported by National Science Foundation Grant DMS-2201085 and by the School of Natural and Behavioral Sciences at Brooklyn College, City University of New York. The research of the third author is partially funded by the Melodia
ANR-20-CE40-0013 project and the 2023-08 Germaine de Sta\"el project. The research of the fourth author is supported by the National Science Foundation under grant DMS-2137661.  Lastly, the research of the fifth author is funded by the Natural Sciences and Engineering Research Council of Canada (NSERC) (funding RGPIN-2019-04844).

\section{Background and  preliminaries}\label{sec:Background}

In this section we recall the basic definitions and results on invariant theory and Artin-Schreier curves that we will require for the work herein.

\subsection{A brief introduction to invariant theory} \label{sec:Invariants} Here we mainly follow \cite{basson, DK02, DK, Dol03, Eis95, Eis05}.
The most important concepts in this section are the definitions of primary and secondary invariants. The main result is Corollary~\ref{cor:reconstruct} that will allow us to manually compute generating sets of invariants in the examples we consider in this paper. For completeness we also introduce the concept of Hilbert series and how to compute them with the Molien-Weyl formula (see Theorem \ref{thm:MF}). This produces a second method for computing invariants that can be use to double-check our computations. Moreover,  the algorithm  implemented in \MAGMA~\cite{magma} to compute invariants is also based on these ideas. Our computations in Section~\ref{sec:DetailedExample} can be corroborated by \MAGMA.

Let $K$ be an algebraically closed field. Let $G$ be a linear algebraic group defined over $K$, acting on an algebraic variety $X$ also defined over $K$. This action defines another action on $K[X]$ by $(g\cdot f)(x)=f(g^{-1}\cdot x)$ for all $x\in X$, $f\in K[X]$, and $g\in G$.

\begin{definition} An element $f\in K[X]$ is an \defi{invariant} for $X$ if $g\cdot f=f$ for all $g\in G$. The algebra of invariants is $K[X]^G\colonequals\{f\in K[X]:\,g\cdot f=f,\forall g\in G\}$.
\end{definition}

We especially consider here the case in which $X=V$ is a rational representation of $G$ of finite degree, i.e.\ a linear representation of finite dimension such that the group morphism $G\to \operatorname{GL}(V)$ is also a morphism of varieties. The $K$-algebras $K[V]$ and $K[V]^G$ are naturally graded $K$-algebras.

In our situation, $G$ will be finite.  Since all finite groups are reductive, we can use the following result due to Emmy Noether \cite{Noether1926}, here using the wording of \cite{DK02}.
\begin{proposition}[Noether Normalization Lemma] \label{lem:noether}
Let $R$ be a finitely generated algebra over a Noetherian commutative ring $K$, and let $G$ be a finite group acting on $R$ by automorphisms fixing $K$ elementwise.  Then $R^G$ is finitely generated as a $K$-algebra.
\end{proposition}

Noether's result is a special case of the following.
\begin{theorem}[Hilbert Finiteness Theorem] If $G$ is a reductive algebraic linear group and $V$ is a rational representation of $G$ of finite degree, then $K[V]^G$ is of finite type. 
\end{theorem}

\begin{definition} Let $A=\oplus_{i\geq 0}A_i$ with $A_0=K$ be a graded $K$-algebra. A set of homogeneous elements $\theta_1,...,\theta_r\in A$ is a \defi{homogeneous system of parameters} (HSOP) if the following hold.
\begin{itemize}
\item $\theta_1,...,\theta_r$ are algebraically independent over $K$,
\item $A$ is a $K[\theta_1,...,\theta_r]$-module of finite type, i.e. there exist $\eta_1,...,\eta_s\in A$ such that
$$
A=\eta_1K[\theta_1,...,\theta_r]+...+\eta_sK[\theta_1,...,\theta_r].
$$
\end{itemize}
When $A=K[V]^G$ and the $\eta_i$ are homogeneous, the elements $\theta_i$ (resp.\ $\eta_i$) are the  \defi{primary (resp. secondary) invariants} of $A$.
\end{definition}

The Noether Normalization Lemma (Prop.\ \ref{lem:noether}) implies that a graded algebra of finite type always admits an HSOP. 

\begin{definition} The \defi{Hilbert series} of a graded $K$-algebra of finite type $A=\oplus_{i\geq0} A_i$ with $A_0=K$ is the power series
$$
H(A,t)=\sum_{i=0}^\infty \operatorname{dim}_K(A_i)t^i. 
$$
\end{definition}

If $A=\eta_1K[\theta_1,...,\theta_r]+...+\eta_sK[\theta_1,...,\theta_r]$, then 
$$
H(A,t)=\frac{\sum_{i=1}^st^{e_i}}{(1-t^{d_1})...(1-t^{d_r})},
$$
where $d_i$ (resp.\ $e_i$) is the degree of $\theta_i$ (resp.\ $\eta_i$).

\begin{theorem}[Molien-Weyl Formula, {\cite[\S~4.6.1]{DK02}}]\label{thm:MF}  The Hilbert series of a rational representation $(V,\rho)$ of finite degree of a compact group $G$ is
$$
H(K[V]^G,t)=\int_G\frac{1}{\operatorname{det}(\operatorname{Id}-\rho(g)t)}d\mu(g),
$$
where $d\mu$ is the Haar measure of $G$.
\end{theorem}

\begin{corollary} If $G$ is a finite group and $\operatorname{char}(K)$ does not divide $|G|$, then 
$$
H(K[V]^G,t)=\frac{1}{|G|}\sum_{g\in G}\frac{1}{\operatorname{det}(\operatorname{Id}-t\cdot g)}.
$$
\end{corollary}

\begin{lemma}[Noether, Fleischmann, Benson, Fogarty, {\cite[Cor. 3.8.4]{DK02}}]
\label{lem:bound} With the previous notation, if $\operatorname{char}(K)$ does not divide $|G|$, then there exists a set of generators of $K[X]^G$, all of degree smaller or equal than $|G|$.
\end{lemma}

These results produce algorithms to compute primary and secondary invariants for actions of finite groups in the non-modular case (i.e.\ $\operatorname{char}(K)=p$ does not divide $|G|$). More precisely, see Algorithms 3.5.4 and 3.7.2 in \cite{DK02}. Moreover, a version of these algorithms is implemented in \MAGMA. They can be extended to work in the modular case as discussed in \cite[\S 3.3 and 3.4.2]{DK02}; specifically in Algorithm 3.7.5 in loc.\ cit.\ and in \cite{DK}. A general bound for the degree of the generators is given in \cite[Cor.\ 0.2]{symonds}. Again, some versions of these algorithms are also implemented in \MAGMA.

\begin{definition} A subset $S\subseteq K[X]^G$ is said to be \defi{separating} if it satisfies the following property. For any two points $x,y\in X$, if there exists an invariant $f\in K[X]^G$ with $f(x)\neq f(y)$, then there exists an element $g\in S$ with $g(x)\neq g(y)$.
\end{definition}

\begin{definition}
Let $A\subseteq K[X]$ be a subalgebra of a polynomial ring of positive characteristic $p$.
Then the algebra
$$
\hat{A}=\{ f\in K[X]:\,f^{p^r}\in A \text{ for some } r\in \mathbb{N}\}\subseteq K[X]
$$
is the \defi{purely inseparable closure} of A in $K[X]$.
\end{definition}

\begin{theorem}[{\cite[Theorem.~2.3.15]{DK02}}] Let $X$ be an affine variety and $G \subseteq \operatorname{Aut}(K[X])$ a subgroup of the automorphisms of the coordinate ring $K[X]$. Then there exists a finite
separating set $S\subseteq K[X]^G$.
\end{theorem}

\begin{theorem}[{\cite[Theorem. 2.3.12]{DK02}}]  Let $G$ be a finite group. Set  $K=\overline{\F}_p$ and let $V$ be a $K$-rational representation of~$G$.
Let $A\subseteq K[V]^G$ be a finitely generated, graded, separating subalgebra. Then $K[V]^G=\hat{\tilde{A}}$, the purely inseparable closure of the normalization of $A$.
\end{theorem}

\begin{corollary}\label{cor:reconstruct}
   Let $X$ be an affine variety over $K=\overline{\F}_p$  and let $G \subseteq \operatorname{Aut}(K[X])$ be a finite subgroup.  
There exists $A\subseteq K[X]^G$ generated by a finite number of invariants from which one can reconstruct a point on $X$ from its values. Moreover, for any such $A$, one has $K[X]^G=\hat{\tilde{A}}$, the purely inseparable closure of the normalization of~$A$.
\end{corollary}

A set of invariants that allow reconstruction of a point as described in Corollary~\ref{cor:reconstruct} is referred to as a \defi{reconstructing system}.

\subsection{Artin-Schreier curves}\label{sec:AScurves}

Throughout, let $p$ be a prime and $\Fpbar$ an algebraically closed field of characteristic~$p$. An \defi{Artin-Schreier curve} is a curve over $\Fpbar$ with an affine model of the form 
\begin{equation}\label{eqn:ASgeneralform}
C_f : y^p-y=f(x),
\end{equation}
where $f(x)\in\Fpbar(x)$ and $f(x) \ne z^p - z$ for any $z \in \Fpbar(x)$. Let $r+1$ (with $r \ge 0)$ be the number of distinct poles of $f(x)$, and denote by $d_i$ be the order of the $i$-th pole of $f(x)$. By \cite[Lemma 3.7.7 (b)]{stichtenoth2009algebraic}, or as we describe in the proof of Theorem \ref{T:normal}, we can assume that no $d_i$ is a multiple of $p$. To simplify formulae, set 
\[ e_i \colonequals d_i+1 \qquad (1 \le i \le r+1) . \]
Then by \cite[Proposition 3.7.8 (d)]{stichtenoth2009algebraic}, the genus of $C_f$ is given by 
\begin{equation}\label{eq:ASgenus}
    g=\frac{p-1}{2}D, \quad \mbox{where} \quad D = -2+\sum_{i=1}^{r+1}e_i. 
\end{equation}
The possible pole counts and orders for $C_f$ thus correspond to partitions of $D+2$ into components $e_i$.

The \emph{$p$-rank} of a smooth, irreducible, projective $\Fpbar$-curve $X$ is the integer $s$ such that the cardinality of the $p$-torsion of its Jacobian $\textrm{Jac}(X)[p](\Fpbar)$ is $p^s$. Then $0\leq s\leq g$, and the Deuring-Shafarevich formula (see \cite{subrao1975p} for a full discussion) implies that for an Artin-Schreier curve $X = C_f$, we have
\begin{equation}\label{eq:ASprank}
    s=r(p-1).
\end{equation}
Thus, the $p$-rank of $C_f$ depends simply on the number of poles of $f(x)$, not the poles themselves or their orders.

Let $\mathcal{AS}_{g}$ denote the moduli space of Artin-Schreier $\Fpbar$-curves of genus $g$ and $\mathcal{AS}_{g,s}$ the locus corresponding to Artin-Schreier $\Fpbar$-curves of genus $g$ with $p$-rank exactly $s$. The following theorem, due to Pries and Zhu~\cite{PriesZhu2012}, characterizes the stratification of $\mathcal{AS}_{g}$ by $p$-rank.

\begin{theorem}[{\cite[Theorem~1.1]{PriesZhu2012}}]\label{theorem:PriesZhuDimension}
    Let $g=D(p-1)/2$ with $D\geq 1$ and $s=r(p-1)$ with $r\geq 0$.
    \begin{enumerate}
        \item The set of irreducible components of $\mathcal{AS}_{g,s}$ is in bijection with the set of partitions $\{e_1, \ldots, e_{r+1}\}$ of $D+2$ into $r+1$ positive integers such that each $e_j\not\equiv 1 \pmod p$.
        \item The irreducible component $\mathcal{AS}_{g,\vec{E}}$ of $\mathcal{AS}_{g,s}$ for the partition $\vec{E}=\{e_1, \ldots, e_{r+1}\}$ has dimension 
        \begin{equation}\label{eq:PriesZhuDimension}
            \dim \mathcal{AS}_{g,\vec{E}}=D-1-\sum_{j=1}^{r+1} \lfloor(e_j-1)/p\rfloor.
        \end{equation}
    \end{enumerate}
\end{theorem}

So the irreducible components of the $p$-rank strata of the moduli space $\mathcal{AS}_{g}$ correspond to the different possibilities for the number of distinct poles of $f(x)$ and their orders, given the constraint in~\eqref{eq:ASgenus}.  Note that an irreducible component $\mathcal{C}$ of $\mathcal{AS}_{g,s}$ is not necessarily an irreducible component of $\mathcal{AS}_g$.  It may be that $\mathcal{C}$ is open in some higher-dimensional irreducible component of $\mathcal{AS}_g$. 

We compute the partitions corresponding to the irreducible components of $\mathcal{AS}_{g,s}$  for small~$g$ in Table~\ref{tab:smallgenustable}.  We omit $p=2$, because all Artin-Schreier curves in this characteristic are hyperelliptic; as mentioned earlier, their invariants and moduli are already well understood.  

\begin{table}[ht]
\centering
\bgroup
\def\arraystretch{1.16}
\begin{tabular}{l||l|l|l||c|c}
$\bm{g}$ & $\bm{p}$ & $\bm{D}$ & $\bm{s}$ & $\bm{\vec{E}}$&$\bm{\dim\mathcal{AS}_{g,\vec{E}}}$    \\ \hline\hline
3   & 3   & 3   & 0   & \{5\}& 1                     \\
    &     &     & 2   & \{3,2\}& 2                   \\
    & 7   & 1   & 0   & \{3\}& 0                     \\ \hline
4   & 3   & 4   & 0   & \{6\}& 2                     \\
    &     &     & 2   & \{3,3\}& 3                   \\
    &     &     & 4   & \{2,2,2\}& 3                 \\
    & 5   & 2   & 0   & \{4\}& 1                     \\
    &     &     & 4   & \{2,2\}& 1 
                   \\ 
\end{tabular} \medskip
\egroup 
\caption{Partitions $\vec{E}$ of $D+2$ corresponding to the irreducible components of $\mathcal{AS}_{g,s}$ and the dimension  of each irreducible component $\mathcal{AS}_{g,\vec{E}}$ for $g = 3, 4$ and $p\geq 3$.}
\label{tab:smallgenustable}
\end{table}

We now consider the isomorphisms of Artin-Schreier curves.
\begin{lemma}\label{lemma:isomorphisms}
Let $C_f$ and $\widetilde{C}_{\tilde{f}}$ be isomorphic Artin-Schreier curves as given in~\eqref{eqn:ASgeneralform}.  Assume that the function field isomorphism $\varphi:\Fpbar(\tilde{C})\to\Fpbar(C)$ fixes the $p$-extension $\Fpbar(C)/\Fpbar(x)$.  Then $\varphi$
must be of the form
\begin{equation}\label{eq:ASisom}
(x, y) \mapsto \left ( \frac{\alpha x + \beta}{\gamma x + \delta}, \lambda y + h(x) \right ) , 
\end{equation} 
where $\alpha, \beta, \gamma, \delta \in \Fpbar$ with $\alpha\delta-\beta\gamma \in \Fpbar^{\times}$, $\lambda \in \F_p^{\times}$, and $h(x) \in \Fpbar(x)$.
\end{lemma}
\begin{proof}
The isomorphism $\varphi$ is completely determined by its respective images $\tilde{x}$ and $\tilde{y}$ of $x$ and $y$. 
By assumption, the degree $p$ extension $\Fpbar(C)/\Fpbar(x)$ is fixed and $\tilde{x}$ is a function of $x$ not dependent on $y$. This implies that $\varphi$ must descend to an isomorphism of $\Fpbar(x)$.  Hence
\begin{equation*}
    \tilde{x}=\varphi(x)=\frac{\alpha x+\beta}{\gamma x+\delta},
\end{equation*}
where $\alpha,\beta,\gamma,\delta\in k $ and $\alpha\delta-\beta\gamma\in\Fpbar^{\times}$. 
To determine the image of $y$ under $\varphi$, 
we note that~$\varphi$ 
is invertible in $\Fpbar(x)(y)$, so it must be of the form 
\begin{equation*}
    \tilde{y}=\varphi(y)=\frac{A(x) y+B(x)}{C(x) y+D(x)},
\end{equation*}
where $A(x),B(x),C(x),D(x)\in \Fpbar(x)$.  By substituting $\tilde x=\frac{\alpha x+\beta}{\gamma x+\delta}$ into the equation of $\tilde{C}_{\tilde{f}}$ and comparing coefficients for $\tilde{y}$, we may assume that $A(x)=\lambda\in\Fpbar^\times$, $C(x)=0$, and $D(x)=1$.
\end{proof}

\begin{corollary}
    \label{cor:generalASisomorphism}
Let $C_f$ be an Artin-Schreier curve as given in~\eqref{eqn:ASgeneralform}.  Assume that $C_f$ is not birational to a curve of the form 
\begin{enumalphii} \itemsep 3pt
    \item $y^p-y=\frac{a}{x^p-x}$, with $a\in\Fpbar$; or \label{case:exceptionalCurve1}
    \item $y^p-y=\frac{1}{x^\lambda}$, with $\lambda \mid p+1$; or \label{case:exceptionalCurve2}
    \item $y^3-y=\frac{i}{x(x-1)}$, with $i^2=2$.\label{case:exceptionalCurve3}
\end{enumalphii}
Then all the isomorphisms of $C_f$ are given by isomorphisms as in~\eqref{eq:ASisom}.
\end{corollary}
\begin{proof}
    By \cite[Theorem~6]{ValentiniMadan}, for $C_f$ not birational to a curve of the form~\eqref{case:exceptionalCurve1}-\eqref{case:exceptionalCurve3}, the cover of $\Fpbar(x)$ of degree $p$ is unique, so we can apply Lemma~\ref{lemma:isomorphisms} to derive the desired result.
\end{proof}

\begin{remark}\label{rem:genusexcep} Henceforth, we will refer to curves that are birational to a curve of the form~\eqref{case:exceptionalCurve1}-\eqref{case:exceptionalCurve3} in Corollary~\ref{cor:generalASisomorphism} as \defi{exceptional}. These curves have genus $(p-1)^2$, $\frac{(p-1)(\lambda-1)}{2}$, and~$2$, respectively.
\end{remark}

\begin{remark}
    The exceptional curve in case~\eqref{case:exceptionalCurve3} of Corollary~\ref{cor:generalASisomorphism} does not appear for the range of parameters considered in this article, so we do not examine it further.  In addition, it is hyperelliptic, and invariants for curves of this type are known (see \cite{LR12}).
    
\end{remark}

\begin{lemma}\label{lemma:explicitextraaut}
    Let $C_f$ be an exceptional Artin-Schreier curve as given in~\eqref{eqn:ASgeneralform}.  Then the automorphisms of $C_f$ are 
    those of type ~\eqref{eq:ASisom} preserving the curve model, together with the automorphisms
\begin{enumalphii}
    \item[(1)] $\tau: \begin{cases} x \mapsto y\\ y \mapsto x \end{cases}$
    \qquad \qquad \qquad \quad \ \ if $y^p-y=\frac{a}{x^p-x}$, with $a\in\Fpbar$; 
    \item[(2)]  
$\phi:\begin{cases} x \mapsto \epsilon xy^{-(p+1)/\lambda)} & \qquad \qquad \mbox{if $y^p-y={x^\lambda}$, with $\lambda \mid p+1$ and $\lambda<p+1$;} \\ y \mapsto 1/y & \qquad \qquad \mbox{where $\epsilon$ is a $\lambda$-root of $-1$;} \end{cases}$

    \item[(2')] $\phi$ in (2) and $\begin{cases}x \mapsto x+a^p\\ y \mapsto y+{a}{x}+b \end{cases}$ if  $y^p-y={x^{p+1}}$, where $a^{p^2}=-a$ and $b^p-b=-a^{p+1}$. 
\end{enumalphii}

\end{lemma}

\begin{proof}
    It is straightforward to see that the isomorphisms from Corollary~\ref{cor:generalASisomorphism} are also isomorphisms for exceptional curves.  The extra function field automorphisms are described in \cite[Theorem~7]{ValentiniMadan} and \cite{Henn} or in \cite[Prop.\ 11.30, Ex.\ A.9]{Torres}. 
\end{proof}

\begin{remark} The automorphism groups of the different cases in Corollary~\ref{cor:generalASisomorphism} and Lemma~\ref{lemma:explicitextraaut} are a semi-direct product of an abelian $p$-group of order $p^2$ with a dihedral group of order $2(p-1)$, so with  cardinality $2p^2(p-1)$, in case (1), an extension of a cyclic group of order $\lambda$ in $PGL(2,p)$ with $\lambda<p+1$, so of cardinality $\lambda|PGL(2,p)|=\lambda (p+1)p(p-1)$, in case (2), and $PGU(3,p^2)$ in case $(2')$ where $|PGU(3,p^2)|=(p^3+1)p^3(p^2-1)$. Finally, in case $(3)$, the automorphism group is a extension of a cyclic group of order $2$ by $S_4$. 
\end{remark}

\subsection{Fractional linear transformations acting on \texorpdfstring{$\Fpbar(x)$}{FFp(x)}}\label{SS:fractionalLinearTransf}
Let \[\varphi(x,y)=\left ( \frac{\alpha x + \beta}{\gamma x + \delta}, \lambda y + h(x) \right ),\] with $\alpha, \beta, \gamma, \delta, \lambda, h(x)$ as given in Corollary~\ref{cor:generalASisomorphism}. For ease of exposition, define 
\[ M\colonequals\begin{pmatrix} \alpha & \beta \\ \gamma & \delta \end{pmatrix}\in \GL_2(\Fpbar) \mbox{ \ and \ } M(x)\colonequals\frac{\alpha x +\beta}{\gamma x + \delta}. \]
Let $f(x)\in\Fpbar(x)$. The map $x\mapsto M(x)$ extends to an action of $\GL_2(\Fpbar)$ on $\Fpbar(x)$ by the extension $M(f(x))=f(M(x))$, which induces an action on the projective line $\P^1(\Fpbar)$ via the poles of functions in $\Fpbar(x)$. When we consider functions with a single pole, this action is the inverse of the well-known action of M\"{o}bius transformations on $\P^1(\Fpbar)$.  For example, consider $f(x)=\frac{1}{x}$, a function with a single pole at $x=0$.  Then $f(M(x))=\frac{\gamma x+\delta}{\alpha x+\beta}$, which has a pole at $x=\frac{-\beta}{\alpha}$ when $\alpha\neq 0$, and a pole at $\infty$ otherwise.  The standard M\"{o}bius transformation of $\P^1(\Fpbar)$ associated to $M$ satifies $M(0)=\frac{\beta}{\delta}$ if $\delta\neq 0$, with $M(0)=\infty$ when $\delta=0$.  However, $M^{-1}(0)=\frac{-\beta}{\alpha}$ when $\alpha\neq 0$, $\infty$ when $\alpha=0$.  

To avoid confusion with these actions, we set the following notation to view poles of functions as points in $\P^1(\Fpbar)$ for the remainder of the paper. Let $P_{\infty}\in\P^1(\Fpbar)$ be the pole of the function $x\in \Fpbar(x)$. For $\mu \in\Fpbar$, let $P_{\mu}\in\P^1(\Fpbar)$ be the pole of the function $\frac{1}{x-\mu}\in \Fpbar(x)$. Throughout, we will describe the action of $\GL_2(\Fpbar)$ on poles using this notation.
For later convenience, we describe the action of $M$ on poles explicitly.

\begin{lemma}\label{lem:pole_images}
The transformation $M(x)=\frac{\alpha x+\beta}{\gamma x + \delta}$ sends the pole $P_{\mu}$ to $P_{\mu^{\prime}}$, where $\mu^{\prime} = M^{-1}(\mu) = \frac{\delta\mu-\beta}{-\gamma\mu+\alpha}$.
\end{lemma}

\section{Standard form}\label{sec:StandardForm}

Every Artin-Schreier curve $C_f$ as given in~\eqref{eqn:ASgeneralform} can be transformed into an isomorphic curve in \defi{standard form}. While this form is not unique, there are a finite number of standard forms for a given isomorphism class, so this form facilitates our invariant computations. Informally, the idea is to find an isomorphism that sends the poles of $C_f$ with the three highest (not necessarily distinct) orders to the pole $P_\infty$ at infinity and the poles $P_0$ and $P_1$ at $x = 0$ and $x = 1$, respectively. In case $C_f$ has at most two poles, they are sent to $P_\infty$ and $P_0$ (for two poles) or simply to $P_\infty$ (for only one pole), and certain normalizations are applied to the polynomial part on the right hand side of~\eqref{eqn:ASgeneralform}. 

The single pole case is the most complicated. Here, the unique pole is moved to infinity, so the right hand side of $C_f$ becomes a polynomial $f(x) \in \Fpbar[x]$ of degree $d$, the order of the pole. We address this case in the following lemma.

\begin{lemma} \label{a1=0}
Let $C_f$ be an Artin-Schreier $\Fpbar$-curve 
as given in~\eqref{eqn:ASgeneralform} with a unique pole at infinity, where $f(x) \in \Fpbar[x]$ 
is monic of degree $d > 1$. Then $C_f$ is isomorphic to a curve of the form $C_g : y^p - y = g(x)$ with
\begin{equation*}
g(x) = x^d + \sum_{i=0}^{d-1} b_i x^i \in \Fpbar[x] ,
\end{equation*}
where $b_i = 0$ whenever $p$ divides $i$ and $b_1 = 0$.
\end{lemma}
\begin{proof}
We explicitly determine $\beta \in \Fpbar$ and $h(x) \in \Fpbar[x]$ such that the isomorphism $(x,y) \mapsto (x+\beta, y+h(x))$ transforms $C_f$ into the required form. 

Write 
\[ f(x) = a_dx^d+\sum_{i=0}^{d-1} a_i x^i \]
with $a_i \in\Fpbar$ for $0 \le i \le d-1$ and $a_d=1$.  For any $\beta \in \Fpbar$ we define
\[u_i(\beta) \colonequals \sum_{j=i}^d a_j \binom{j}{i} \beta^{j-i} \ (0 \le i \le d),\]
so $f(x+\beta) = \sum_{i=0}^d u_i(\beta) x^i$ and $u_i(0) = a_i$ for $0 \le i \le d$.
Set $M = \lfloor d/p \rfloor$ and $m = \lfloor M/p \rfloor$, where $\lfloor \cdot \rfloor$ denotes the floor function. Then $Mp < d$, as $p$ does not divide $d$. 
If $M > 0$, recursively define elements $b_i(\beta) \in \Fpbar$ via
\[ b_i(\beta)^p = \begin{cases} 
		u_{ip}(\beta) & \mbox{for $M \ge i \ge m+1$}, \\ 
		u_{ip}(\beta) + b_{ip}(\beta) & \mbox{for $m \ge i \ge 1$} . \end{cases} \]
Since $i < ip \le mp \le M$ for $1 \le i \le m$, the quantities $b_i(\beta)$ are well defined. 
Also define $b_0(\beta) \in \Fpbar$ via the identity 
\[ b_0(\beta)^p - b_0(\beta) = u_0(\beta) . \]
Let $f'(x)$ denote the formal derivative of $f(x)$ with respect to $x$. Since $p$ does not divide $d$, the degree of $f'(x)$ is $d-1 > 0$, and we note that $u_1(\beta) = f'(\beta)$. 

Now choose $\beta \in \Fpbar$ such that $u_1(\beta) + b_1(\beta) = 0$. For this choice of $\beta$, define 
\[ h(x) = \sum_{i=0}^M b_i(\beta) x^i . \]
Then the isomorphism $(x,y) \mapsto (x + \beta, y + h(x))$ maps $C_f$ to a curve of the form $C_g : y^p - y = g(x)$ where
\[ g(x) = f(x+\beta) - h(x)^p + h(x) \in \Fpbar[x] . \]
Since $h(x)^p$ has degree $Mp < d$ and $f(x+\beta)$ is monic of degree $d$, $g(x)$ is also monic of degree $d$. It remains to show that the coefficients of $g(x)$ satisfy the desired conditions.

Every monomial of the form $x^{ip}$ potentially appearing in $g(x)$ satisfies $0 \le i \le M$. For $i \ge 1$, the corresponding coefficient of $g(x)$ is
\[ b_{ip}(\beta) = \begin{cases} u_{ip}(\beta) - b_i(\beta)^p & \mbox{for $M \ge i \ge m+1$}, \\
		u_{ip}(\beta) - b_i(\beta)^p + b_{ip}(\beta) & \mbox{for $m \ge i \ge 1$}. \end{cases} \]
Moreover, $b_0(\beta) = u_0(\beta) - b_0(\beta)^p + \beta_0(\beta)$. By the definition of $b_i(\beta)$, we have $b_{ip} = 0$ for $0 \le i \le M$. Finally, the coefficient of $x$ in $g(x)$ is $b_1(\beta) = u_1(\beta) + b_1(\beta)$, which also vanishes by our choice of $\beta$.
\end{proof}

\begin{remark}
    
The normalization conditions proposed in \cite[Prop.\ 2.1.1]{farnell} require the removal of the monomial $x^{d-1}$ as well as all $p$-th power monomials in $f(x)$ via an iterative sequence of suitable isomorphisms. One difficulty arising with these restrictions is that in the case where $d \equiv 1 \pmod{p}$, removal of $p$-th powers already eliminates the monomial $x^{d-1}$, in which case it is desirable to remove a second coefficient in order to further simplify the curve model. It is in fact always possible to eliminate at least one of the next two highest order monomials $x^{d-2}$, $x^{d-3}$, but the form of the initial curve determines which of them can be removed. Specifically, all linear transformations on $x$  leave the coefficient of $x^{d-2}$ in $f(x)$ fixed, so this monomial can be eliminated precisely when the corresponding term in the initial curve vanishes. In order to avoid this dependence on the shape of the curve, we opt instead to remove the linear term and all $p$-th power monomials from $f(x)$ at once, with a single isomorphism. This choice of normalization has the advantage of being more canonical and thus mathematically more satisfying. However, in practice, it may come at the expense of complicating the calculation of the curve invariants compared to a model with fewer high degree monomials; the example in Section~\ref{SS:one-pole-order-5} illustrates this. 
\end{remark}

We now have all the ingredients to convert an Artin-Schreier curve to standard form.

\begin{theorem} \label{T:normal}
Let $p$ be an odd prime and $C_f$ an Artin-Schreier $\Fpbar$-curve as given in~\eqref{eqn:ASgeneralform} with $r+1$ poles of respective orders $d_1 \ge d_2 \ge \ldots \ge d_{r+1}$. Then $C_f$ is isomorphic to an Artin-Schreier curve 
\[ C_g: y^p - y = g(x), \] 
where $g(x) \in \Fpbar(x)$ takes on one of the following forms:

\begin{enumerate}
\item \label{case:r=0}\textbf{Case $r = 0$}: 
\[ g(x) = x^{d_1} + Q(x) \]
where 
$Q(x) \in \Fpbar[x]$ is a multiple of $x^2$ and no monomial appearing in $Q(x)$ has an exponent that is divisible by $p$. 

\item \label{case:r=1}\textbf{Case $r = 1$}: 
\[ g(x) = F(x) + G \left ( \frac{1}{x} \right ), \]
where $F(x), G(x) \in \Fpbar[x]$, $F(x)$ is monic, $\deg(F) = d_1$, $\deg(G) = d_2$, and no monomial appearing in $F(x)$ or $G(x)$ has an exponent 
that is divisible by $p$.

\item \label{case:r>=2}\textbf{Case $r \ge 2$}: 
\[ g(x) = F(x) + G \left (\frac{1}{x} \right ) + H \left ( \frac{1}{x-1} \right ) + S(x), \]
where $F(x), G(x), H(x) \in \Fpbar[x]$, $\deg(F) = d_1$, $\deg(G) = d_2$, $\deg(H) = d_3$, either $S(x) = 0$ or 
\[ S(x) = \sum_{i=4}^{r+1} \frac{g_i(x-\lambda_i)}{(x-\lambda_i)^{d_i}}, \] 
with $\lambda_i \in \Fpbar \setminus \{ 0, 1 \}$, $g_i(x) \in \Fpbar[x]$ non-zero, $\deg(g_i) < d_i$, and no monomial appearing in $F(x)$, $G(x)$, $H(x)$, or any of the polynomials $x^{d_i} g_i(x^{-1})$ has an exponent that is divisible by~$p$, for $4 \le i \le r+1$.
\end{enumerate}
The curve $C_g$ is said to be in a \defi{standard form}.
\end{theorem}
\begin{proof}
We represent fractional linear transformations $M(x) = \displaystyle \frac{\alpha x + \beta}{\gamma x + \delta}$ by $\begin{pmatrix} \alpha & \beta \\ \gamma & \delta \end{pmatrix}$ as in Section~\ref{SS:fractionalLinearTransf}.

We begin with the case $r = 0$, so $C_f$ has a unique pole $P_{\mu_1}$ of order $d_1$. If $\mu_1 \ne \infty$, then $\mu_1 \in \Fpbar$ and $f(x)$ is of the form
\begin{equation*} 
f(x) = \frac{f_\infty(x-\mu_1)}{(x-\mu_1)^{d_1}} ,
\end{equation*}
where $f_\infty(x) \in \Fpbar[x]$ is non-zero of degree less than $d_1$. The matrices 
\begin{alignat*}{2}
& \begin{pmatrix} 1 & 0 \\ 0 & 1 \end{pmatrix} && \quad \mbox{if $\mu_1 = \infty$}, \\ 
& \begin{pmatrix} \mu_1 & -1 \\ 1 & 0 \end{pmatrix} && \quad \mbox{if $\mu_1 \in \Fpbar$},
\end{alignat*}
send $P_{\mu_1}$ to $P_\infty$, so we produce a curve of the form $y^p - y = \tilde{f}(x)$, where $\tilde{f}(x) \in \Fpbar[x]$ is a polynomial of degree $d_1$. 
Let $a$ be the leading coefficient of $\tilde{f}$.
Applying a matrix of the form $\begin{pmatrix} a^{-1/p} & 0 \\ 0 & 1 \end{pmatrix}$ yields the curve $y^p - y = \tilde{g}(x)$ where $\tilde{g}(x) = \tilde{f}(a^{-1/d}x)$ 
is monic of degree $d_1$. 

If $d_1 = 1$, then define $\gamma \in \Fpbar$ via $\gamma^p - \gamma = \tilde{g}(0)$, the constant coefficient of $\tilde{g}(x)$. Then the isomorphism $(x,y) \mapsto (x, y+\gamma)$ produces the curve $y^p - y = x = x^{d_1}$. If $d_1 > 1$, then the isomorphism of Lemma~\ref{a1=0} produces a curve $C_g$ of the required form. 

For the remaining cases (i.e.\ $r \ge 1$), the conversion of $C_f$ to standard form is more easily done in two stages. The first stage applies a suitable fractional linear transformation of $x$ to $C_f$ that moves the highest order poles as explained earlier. This is followed by a second sequence of linear transformations applied to $y$ that eliminates $p$-th powers of $x$ one-by-one. 

Assume first that $r = 1$, so $C_f$ contains two poles $P_{\mu_1}, P_{\mu_2}$ of respective orders $d_1 \ge d_2$. Then the matrices
\begin{align*}
\begin{pmatrix} 1 & \mu_2 \\ 0 & 1 \end{pmatrix}     & \quad \mbox{if $\mu_1 = \infty$ and $\mu_2 \in \Fpbar$}, \\
\begin{pmatrix} \mu_1 & 1 \\ 1 & 0  \end{pmatrix}    & \quad \mbox{if $\mu_1 \in \Fpbar$ and $\mu_2 = \infty$}, \\
\begin{pmatrix} \mu_1 & \mu_2 \\ 1 & 1 \end{pmatrix} & \quad \mbox{if $\mu_1 \in \Fpbar$ and $\mu_2 \in \Fpbar$},
\end{align*}
send $P_1$ and $P_2$ to $P_\infty$ and $P_0$, respectively, and thus produce curve of the form $y_p - y = \tilde{f}(x)$, where $\tilde{f}(x) \in \Fpbar[x, x^{-1}]$ is a Laurent polynomial of degree $d_1$ in $x$ and degree $d_2$ in $x^{-1}$. Let~$a$ be the coefficient of $x^{d_1}$ in $\tilde{f}(x)$. As in the case $r=0$, a scaling matrix of the form
\[ \begin{pmatrix} a^{-1/p} & 0 \\ 0 & 1 \end{pmatrix}  \]
yields an isomorphic curve $C_{\tilde{g}}$ where $\tilde{g}(x)$ is monic with respect to $x$. 

The case $r = 2$ proceeds similarly. Suppose $C_f$ has three or more poles. Let $P_{\mu_1}, P_{\mu_2}$, $P_{\mu_3}$ be poles of the three largest resspective orders $d_1 \ge d_2 \ge d_3$. Then the matrices
\begin{alignat*}{2}
&\begin{pmatrix} \mu_3 - \mu_2 & \mu_2 \\ 0 & 1 \end{pmatrix} 
		&& \quad \mbox{if $\mu_1 = \infty$, $\mu_2 \in \Fpbar$ and $\mu_3 \in \Fpbar$}, \\
&\begin{pmatrix} \mu_1 & \mu_1-\mu_3 \\ 1 & 0  \end{pmatrix} 
		&& \quad \mbox{if $\mu_1 \in \Fpbar$, $\mu_2 = \infty$ and $\mu_3 \in \Fpbar$}, \\
&\begin{pmatrix} -\mu_1 & \mu_2 \\ -1 & 1 \end{pmatrix} 
		&& \quad \mbox{if $\mu_1 \in \Fpbar$, $\mu_2 \in \Fpbar$ and $\mu_3 = \infty$}, \\
&\begin{pmatrix} \mu_1(\mu_3-\mu_2) & \mu_2(\mu_1-\mu_3) \\ \mu_3-\mu_2 & \mu_1-\mu_3 \end{pmatrix} 
		&& \quad \mbox{if $\mu_i \in \Fpbar$ for $i = 1, 2, 3$},
\end{alignat*}
send $P_1, P_2$, and $ P_3$ to $P_\infty$, $P_0$, and $P_1$, respectively. The same scaling matrix as in the case $r = 1$ produces a rational function $\tilde{g}(x)$ of the specified form that is monic in $x$.

Next, in both of the cases $r=1$ and $r\ge 2$, we remove all the monomials of the form $x^{kp}$ for $k\in\Z\setminus0$. 
The constant coefficient of $\tilde{g}(x)$ is handled separately at the end. 

As described in \cite[Prop.\ 2.1.1]{farnell} and the subsequent discussion, removal of the monomials $x^{kp}$ with $k > 0$ from $\tilde{g}(x)$ is accomplished iteratively via suitable translations of $y$, leaving~$x$ fixed. Suppose $\tilde{g}(x)$ contains a term of the form $bx^{pk}$, where $b \in \Fpbar^{\times}$ and $1 \le k \le d_1-1$. Then the isomorphism $(x, y) \mapsto (x,y+b^{1/p}x^k)$  
produces the curve $y^p - y = \tilde{g}(x) - b x^{kp} - e^{1/p}x^k$, where the term $bx^{pk}$ no longer appears on the right hand side. 
Systematically applying a finite number of suitable isomorphisms of this form, looping over the values of $k$ in decreasing order from $k = \lfloor d_1/p \rfloor$ to $k = 1$, eliminates 
all terms monomials $x^{kp}$ with $k > 0$.  An analogous process, looping over all $k$ from $-\lfloor (d_2-1)/p \rfloor$ to $-1$, eliminates all monomials $x^{kp}$ with $ k < 0$. 

When $r=1$, this process removes all monomials that are positive $p$-powers. So suppose that $r \ge 2$. Then for each pole $P_\mu$ with $\mu = 1$ or $\mu = \lambda_i$ $(4 \le i \le r+1)$, we apply the isomorphisms $(x,y) \mapsto (x,y + b(x-\mu)^{-kp})$, where $k$ runs from $-\lfloor (d_i-1)/p \rfloor$ to $-1$ for $3 \le i \le r+1$, to remove all terms involving powers $(x-\mu)^{kp}$ with $k < 0$. 

We are at last left with a curve of the form  $y^p - y = \tilde{g}(x)$, where no terms $x^{pk}$ and $(x-\mu)^{-kp}$ for any pole $P_\mu \ne P_{\infty}$, with $k > 0$, appear in $\tilde{g}(x)$.  It remains to eliminate the constant term $\tilde{g}(0)$. To that end, let $\gamma \in \Fpbar$ such that $\gamma^p - \gamma = \tilde{g}(0)$. Then the isomorphism $(x,y) \mapsto (x,y+\gamma)$ yields the curve $y^p - y = g(x)$, where $g(x) = \tilde{g}(x) - \tilde{g}(0)$ is of the desired form.    
\end{proof}

As mentioned above, the standard form of an Artin-Schreier curve is not unique. For example, the two standard form curves $y^3 - y = x^4 - x^2$ and $y^3 - y = x^4 + x^2$ are $\overline{\F}_3$-isomorphic via the isomorphism $(x,y) \mapsto (\sqrt{2}x,y)$; this is a special case of Proposition~\ref{prop:g3isomorphism_ex1}. However, for any non-exceptional curve, we see that the number of possible standard forms is finite and that the forms are easy to enumerate. In fact, the number of variable coefficients in each standard form is equal to the dimension of the component of the moduli space with the corresponding partition as described in Theorem~\ref{eq:PriesZhuDimension}.  

In the next section, we discuss isomorphisms between Artin-Schreier curves in standard form.

\subsection{Isomorphisms Between Curves in Standard Form} \label{subsec: isomorphisms}

If we consider only Artin-Schreier curves in standard form, outside of the exceptional curves in Corollary~\ref{cor:generalASisomorphism}, this narrows the possibilities for isomorphisms between them. Let 
\begin{equation}\label{eq:generalIsomorphism}
\varphi(x,y)=\left ( M(x), \lambda y + h(x) \right ),
\end{equation}
with $M, \lambda, h(x)$ as in Corollary~\ref{cor:generalASisomorphism}. 

\begin{lemma}\label{lem:uniqueh}
    Let $C_f$ and $C_{\tilde{f}}$ be isomorphic Artin-Schreier curves in standard form. Then for every choice of $M\in\GL_2(\Fpbar)$ for an isomorphism $\varphi(x,y)$ as in~\eqref{eq:generalIsomorphism} between these curves, there is a unique choice of $h(x)$ up to a constant in $\mathbb{F}_p$, that is, up to (multiple) composition with the isomorphism $\sigma: (x,y)\mapsto(x,y+1)$.
\end{lemma}

\begin{proof}
    We apply $\varphi$ to the equation $y^p-y=f(x)$ to obtain $f(x)=\frac{1}{\lambda}\left(f(M(x))-h^p(x)+h(x)\right)$. This enforces a unique choice of $h(x)$ up to a constant in $\mathbb{F}_p$.
\end{proof}

For $M\in\GL_2(\Fpbar)$, we define $h_{M}(x)\in \Fpbar(x)$ to be a choice of polynomial from Lemma~\ref{lem:uniqueh}, up to composition with powers of $\sigma$. An isomorphism between standard forms is then determined, up to composition with powers of $\sigma$,
by the choice of $M\in GL_2(\Fpbar)$ and $\lambda\in\F_p^{\times}$. Thus we may without confusion define 
\begin{equation}\label{eq:standardASisom}
\varphi_{\lambda,M}(x,y)=(M(x),\lambda y +h_M(x)).
\end{equation}

In our application, we consider rational functions $f(x)$ which appear in the standard form of curves $y^p-y=f(x)$ corresponding to a given prime $p$ and partition $\vec{E}=\{e_1, e_2, \dots, e_{r+1}\}$ of $D+2$.  As before, we assume that $e_{i+1}\leq e_i$ for $1\leq i\leq r$, where $d_i=e_i-1$ for each $i$. Since genus and $p$-rank are invariant under isomorphism, isomorphic curves must correspond to the same partition $\vec{E}$. The isomorphism from one curve to another preserves the orders of distinct poles, but may change their locations. 

The possible isomorphisms between non-exceptional Artin-Schreier curves in standard form can be determined by finding all $M\in \GL_2(\Fpbar)$ and $\lambda\in\F_p^{\times}$ that preserve all the restrictions on poles and curve coefficients imposed by Theorem~\ref{T:normal} as well as  the partition~$\vec{E}$. In particular, subject to these conditions, subsets of poles may be permuted and curve coefficients changed by such an isomorphism. We designed standard forms in such a way that the number of free coefficients of $f(x)$ in the right hand side of a standard form curve~\eqref{eqn:ASgeneralform} is equal to the dimension $d$ of the corresponding irreducible component of $\mathcal{AS}_{g,s}$.  Thus, the group $G$ of isomorphisms between standard forms must be finite, and we are able to completely enumerate the images of non-exceptional curves under $G$. Let $a_1,a_2,\dots,a_d$ be the free coefficients of the standard form of some irreducible component.  We can think of $\Fpbar[a_1,a_2,\dots, a_d]$ as the space containing all standard form models of curves in $\mathcal{AS}_{g,s}$.  The isomorphisms in $G$ act on $\Fpbar[a_1,a_2,\dots, a_d]$, and we seek a separating, or ideally, a reconstructing subset of the invariant ring $\Fpbar[a_1,a_2,\dots, a_d]^G$. Finding a minimal reconstructing set of invariants allows us to essentially parameterize an irreducible component of $\mathcal{AS}_{g,s}$.

\begin{remark} As just mentioned, the number of free coefficients of a standard form is equal to the dimension of the corresponding irreducible component of $\mathcal{AS}_{g,s}$. This number is also equal to the number of primary invariants for the action of the finite group $G$ acting on the coefficients of the standard form. 
\end{remark}

\subsection{Standard forms of exceptional curves}

In Corollary~\ref{cor:generalASisomorphism} we described the isomorphisms from an Artin-Schreier curve $C$ when the function field $\overline{\mathbb{F}}_p(C)$ has a single index $p$ subextension associated to a genus $0$ curve. Such a subextension gives a model of the curve of the form $y^p-y=f(x)$ and we discussed its standard forms in Section~\ref{subsec: isomorphisms}. We now investigate what standard forms can arise for the cases in which such an index $p$ subextension is not unique. 
For each order $p$ automorphism $\sigma$ of $C$ such that the quotient curve $C/\langle\sigma\rangle$ has  genus $0$, the degree $p$ extension $\overline{\mathbb{F}}_p(C)/\overline{\mathbb{F}}_p(C/\langle\sigma\rangle)$ produces a way to write $C$ as $y^p - y = f(x)$ where $y\in \overline{\mathbb{F}}_p(C)\setminus \overline{\mathbb{F}}_p(C/\langle\sigma\rangle)$ and $f(x) \in \overline{\mathbb{F}}_p(C/\langle\sigma\rangle)$.

\begin{lemma}\label{lem:standard1}
Let $C:\,y^p-y=\frac{a}{x^p-x}$ be a curve as given in case~\eqref{case:exceptionalCurve1} of Corollary~\ref{cor:generalASisomorphism}. Then any standard form of $C$ can be obtained by applying an isomorphism of the form in~\eqref{eq:ASisom} through the procedure in Theorem~\ref{T:normal} from the given model of $C$.
\end{lemma}

\begin{proof} The automorphism group of $C$
is isomorphic to $((C_p\times C_p)\rtimes C_{p-1})\rtimes C_2$, generated by 
\begin{align*}
(1,0,0,0)(x,y)&=(x+1,y), \\
(0,1,0,0)(x,y)&=(x,y+1), \\
(0,0,1,0)(x,y)&=(\mu x,y/\mu) \mbox{ \ with $\mu \in \Fpbar^\times$ and} \\
(0,0,0,1)(x,y)&=(y,x).
\end{align*}
Any subgroup of order $p$ consists of elements of the form $(\alpha,\beta,0,0)$ and is generated by 
an automorphism of the form $(0,1,0,0)$ or $(1,\beta,0,0)$. The first of these does not change the curve equation.
The second one produces a degree $p$ extension $\overline{\mathbb{F}}_p(x,y)/\overline{\mathbb{F}}_p(y-\beta x, y^p-y)$ that is isomorphic to $\overline{\mathbb{F}}_p(\mathbb{P}^1)$ if $\beta=0$; if $\beta\neq 0$, it defines the hyperelliptic curve of genus $p-1$ given by $v^p-v=\frac{a}{u}-\beta u$. So once again we obtain the original equation $x^p-x=\frac{a}{y^p-y}$ and the result holds.
\end{proof}

A result from \cite{Konto} explicitly describes the isomorphism classes of these curves in terms of the non-standard model.
\begin{lemma}\cite[Lemma 3.1]{Konto}\label{lem:konto}
    Let $a, a^{\prime}\in\Fpbar^{\times}$.  Then $y^p-y=\frac{a}{x^p-x}$ is isomorphic to $y^p-y=\frac{a^{\prime}}{x^p-x}$ if and only if $a=\lambda a^{\prime}$ for some $\lambda\in\F_p^{\times}$. 
\end{lemma}

\begin{lemma}\label{lem:standard2} 
Let $C:\,y^p-y=\frac{1}{x^\lambda}$. with $\lambda\mid p+1$, be a curve as given in case~\eqref{case:exceptionalCurve2} of Corollary~\ref{cor:generalASisomorphism}.  Then any standard form of $C$ can be obtained by applying an isomorphism of the form in~\eqref{eq:ASisom} through the procedure of Theorem~\ref{T:normal} from the given model of $C$. In particular, the only standard form of $C$ is $y^p-y={x^\lambda}$.
\end{lemma}

\begin{proof}
The automorphisms of this curve are described in 
Lemma~\ref{lemma:explicitextraaut}. When $\lambda\neq p+1$, the order $p$ subgroups are generated by 
\[ (x,y)\mapsto(x,y+1) \mbox{ \ and \ } (x,y)\mapsto \left ( \frac{x}{(y+1)^{(p+1)/\lambda}}, \frac{y}{y+1} \right ). \]
They all produce curves of the form $y^p-y=x^\lambda$. When $\lambda=p+1$, there are extra order $p$ subgroups. Rewriting the equation as $y^p-y=x^{p+1}$, it suffices to consider the automorphisms $(x,y)\mapsto(x+a^p,y+ax+b)$ with $a^{p^2}+a=0$ and $b^p-b-a^{p+1}=0$, producing the degree~$p$ subextension $\overline{\mathbb{F}}_p(x^p+x,P)$ where $P$ is the product of all the images of $y$ under these automorphisms. This subextention is not isomorphic to $\overline{\mathbb{F}}_p(x)$, and hence neither to $\overline{\mathbb{F}}_p(\mathbb{P}^1)$. If it were, $C$ would be birational to $x^p+x=0$,  but this curve is $p$ copies of  $\mathbb{P}^1$ and hence reducible.
\end{proof}
\begin{remark}
    The upshot of Lemmas~\ref{lem:standard1}
    and~\ref{lem:standard2} is that there are only a finite number of standard forms for each isomorphism class of exceptional curves, and they can be easily enumerated. This confirms again that the group of isomorphisms between exceptional curves in standard form is finite.
\end{remark}

\section{Invariant Computations}
\label{sec:DetailedExample} 
In this section we determine invariants for curves within each stratum of $\mathcal{AS}_{g}$ for Artin-Schreier curves of genus $g=3,4$ in odd characteristic.  As we see in Table~\ref{tab:smallgenustable}, there is a single irreducible component in each stratum in these cases, so each stratum has a unique system of invariants.

\subsection{Genus 3, characteristic 3}
In this case, the quantity $D$ defined in~\eqref{eq:ASgenus} takes on the value $D=3$.
There are two partitions of $D+2=5$ (up to reordering) that satisfy the conditions of Theorem~\ref{theorem:PriesZhuDimension}, namely
\begin{equation}\label{eq:Partitions_p3g3}
    \{5\} \hspace{.2in}\text{and}\hspace{.2in} \{3,2\},
\end{equation}
so there are two strata in $\mathcal{AS}_{g}$. These partitions also determine the standard forms of all possible genus 3 Artin-Schreier curves in characteristic 3. Recall that $d_i=e_i-1$ is the order of the pole $P_{\mu_i}$ of $f(x)$ for the Artin-Schreier curve $C_f\colon y^3-y=f(x)$. The partitions in~\eqref{eq:Partitions_p3g3} show that $C_f$ has either one pole of order 4 or two poles with respective orders 2 and 1. We treat these two cases separately below.

\subsubsection{One pole of order 4}
We use~\eqref{eq:PriesZhuDimension} in Theorem~\ref{theorem:PriesZhuDimension} with $D=3$, $r = 0$ and $e_1 = 5$ to determine that the dimension of this stratum is 
$$\dim \mathcal{AS}_{3,\{5\}} = 
3 - 1 - \lfloor 4/3 \rfloor = 1.$$
We have $r=0$ and $d_1=4\equiv 1\pmod{3}$, corresponding to Case~\eqref{case:r=0} of Theorem~\ref{T:normal} and yielding a standard form 
\begin{equation}\label{eq:1poleOrder4}
   C: y^3-y=x^4+ax^2,
\end{equation} with $a \in\overline{\F}_3$.
When $a=0$, this curve is isomorphic to an exceptional curve of type~\eqref{case:exceptionalCurve2} in Corollary~\ref{cor:generalASisomorphism} which has standard form $y^3-y=x^4$.  This is the only exceptional curve in the family, see Lemma~\ref{lem:standard2}.   

\begin{proposition}\label{prop:g3isomorphism_ex1}
 Every isomorphism between curves in standard form as in~\eqref{eq:1poleOrder4} with $a\ne 0$ is given, up to composition with powers of $\sigma:\,(x,y)\mapsto(x,y+1)$, by
\begin{equation}\label{eq:isomorphiamp3g3a}
    (x,y)\mapsto (\alpha x, \lambda y),
\end{equation}
where $\lambda\in \F_3^{\times}$ and $\alpha\in \overline \F_3$ with $\alpha^4=\lambda$.
\end{proposition}
\begin{proof}   Let $C$ be as in equation~\eqref{eq:1poleOrder4}.  By Corollary~\ref{cor:generalASisomorphism}, disregarding composition with powers of $\sigma$, every isomorphism $\varphi$ of $C$ is given by $\varphi_{\lambda,M}$ as in~\eqref{eq:standardASisom} with $M=\begin{pmatrix}
\alpha & \beta \\
\gamma & \delta
\end{pmatrix}\in \GL_2(\overline{\F}_3)$.
Since the image of $C$ must be in standard form, $\varphi_{\lambda,M}$ must not move the pole at $\infty$, so $\gamma=0$ and we can assume without loss of generality that $\delta=1$.  Since the coefficient of $x$ in $\varphi_{\lambda,M}(C)$ must vanish, we conclude that $\beta=0$, and that the unique polynomial $h(x)$ that produces a curve in standard form is $h(x)=0$.  So $\varphi_{\lambda,M}$ is of the form $(x,y)\mapsto (\alpha x, \lambda y)$.
and yields the curve
$$\lambda(y^3-y)=\alpha^4 x^4+a\alpha^2x^2.$$ Since the right hand side must be monic in the standard form, we have $\alpha^4=\lambda.$ 
\end{proof}

In the language of group actions, we now have a finite group $G\simeq \mathbb{Z}/4\mathbb{Z}$ acting on $\overline{\F}_3[a]$.  This action is via isomorphisms of curve models of the form $y^3-y=x^4+ax^2$, sending $a\mapsto \epsilon a$ with $\epsilon^4=1$. We see that $G$ acts linearly on $\overline{\F}_3[a]$.

\begin{corollary}
The element $I_1= a^4$ 
is an invariant and generates the ring of invariants for Artin-Schreier curves of genus $3$ in characteristic $3$ with $3$-rank $0$,
i.e.\ $\overline{\F}_3[a]^G=\overline{\F}_3[I_1]$.\end{corollary}

\begin{proof}
It is straightforward to check that $I_1$ is an invariant for the action of $G$ and forms a reconstructing set for the family (namely, different choices of $a$ such that $a^4=I_1$ for $y^3-y=x^4+ax^2$ produce isomorphic curves in standard form). By Corollary~\ref{cor:reconstruct} we have $\overline{\F}_3[a]^G=\widehat{\widetilde{\overline{\F}_3[I_1]}}=\overline{\F}_3[I_1]$. 
\end{proof}

\begin{remark} This corollary can also be proved by computing a set of generators for $\overline{\F}_3[a]^G$ using \MAGMA; or by manually checking for invariants up to degree 4 (see Lemma~\ref{lem:bound}).
\end{remark}

\begin{remark} If we start with a model $y^3-y=ax^4+bx^3+cx^2+dx+e=p(x)$, the invariant $I_1$ (of a curve in the family $y^3-y=x^4+ax^2$ isomorphic to the starting curve) is given by $c^4/a^2$. More generally, if we start with $y^3-y=\frac{ax^4+bx^3+cx^2+dx+e}{(x-\lambda)^4}$, the reconstructing invariant can be chosen to be $I_1=c^4/(a\lambda^4+b\lambda^3+c\lambda^2+d\lambda+e)^2$.
\end{remark}

\subsubsection{One pole of order 2 and one pole of order 1}
We use~\eqref{eq:PriesZhuDimension} in Theorem~\ref{theorem:PriesZhuDimension} with $D = 3$, $r = 1$, $e_1 = 3$ and $e_2 = 2$ to determine that the dimension of this stratum is 
$$\dim \mathcal{AS}_{3, \{3,2\}} = 
3-1-\lfloor 2/3 \rfloor - \lfloor 1/3 \rfloor = 2.$$
Based on Theorem~\ref{T:normal},
the standard form of curves with a pole of order 2 and a second pole of order 1 is
\begin{equation}\label{eqn:PolesOfOrder1And2}
    C\colon y^3-y=x^2+ax+\frac{b}{x}
\end{equation}
for $a, b \in \overline{\mathbb{F}}_3$ with $b \ne 0$.

\begin{proposition}\label{prop:g3isomorphism_ex2}
Every isomorphism between curves in standard form as in~\eqref{eqn:PolesOfOrder1And2} is given, up to composition with powers of $\sigma:\,(x,y)\mapsto(x,y+1)$, by
\begin{equation}\label{eq:isomorphiamp3g3b}
    (x,y)\mapsto (\alpha x, \lambda y),
\end{equation}
where $\lambda\in \F_3^{\times}$ and $\alpha\in \overline \F_3$.
\end{proposition}
\begin{proof}
No curve in this family is exceptional, so we can directly use Corollary~\ref{cor:generalASisomorphism} to see that, disregaring composition with powers of $\sigma$, every isomorphism is of the form $\varphi_{\lambda,M}$ as in~\eqref{eq:standardASisom} with $M=\begin{pmatrix}
\alpha & \beta \\
\gamma & \delta
\end{pmatrix}\in \GL_2(\overline{\F}_3)$. 
In this particular case, each isomorphism must fix~$\infty$ and~$0$, so $\gamma=\beta=0$, and we can choose $\delta$ to be equal to 1.  This implies that $h(x)=0$, so the only isomorphisms of $C$ as in~\eqref{eqn:PolesOfOrder1And2} are of the form $(x,y)\mapsto (\alpha x, \lambda y)$.  Applying the isomorphism to $C$ and using the requirement that the right hand side of a standard form must be monic yields $\alpha^2=\lambda$, so there are four possible values for $\alpha$ since $\lambda\in\F_3^{\times}$. This produces the model
\begin{equation}\label{eq:IsoClassp3g3b}
    \widetilde{C}\colon y^3-y=x^2+a'x+\frac{b'}{x},
\end{equation}
where $a'=a/\alpha$ and $b'=b/\alpha^3$. Since $a$ and $b$ can take on any values in $\overline\F_3$, we obtain a two-dimensional family of curves from the isomorphism in~\eqref{eq:isomorphiamp3g3b}. Since the stratum is two-dimensional, this isomorphism yields all curves with the prescribed poles and orders.
\end{proof}

These isomorphisms give a group action of the finite group $G\simeq \Z/4\Z$ acting linearly on $\overline{\F}_3[a,b]$.  The action is generated by $(a,b)\mapsto(ia,-ib)$, where $i^2=-1$.

\begin{corollary}
The elements $I_1= a^4, I_2= ab,$ and $I_3= b^4$ generate the ring of invariants for Artin-Schreier curves of genus $3$in characteristic $3$ with $3$-rank $2$, i.e.\ $\overline{\F}_3[a,b]^G=\overline{\F}_3[I_1,I_2,I_3]$. These elements satisfy the algebraic relation $I_1I_3-I_2^4=0$.
\end{corollary}

\begin{proof}
Let $a'$ and $b'$ denote the coefficients of $1/x$ and $x$ in~\eqref{eq:IsoClassp3g3b}, respectively. Then $(a')^4=(a/\alpha)^4=a^4$ and $(b')^4=(b/\alpha^3)^4=b^4$ for all $\alpha\in\overline{\F}_3$. Hence $I_1$ and $I_3$ are invariants. Furthermore, $a'b'=ab/\alpha^4=ab,$ so $I_2$ is also an invariant.  On the other hand, given values $\{I_1,I_2,I_3\}$ and choosing $a,b\in\overline{\F}_3$ with $a^4=I_1$, $b^4=I_3$ and $ab=I_2$, we can reconstruct a unique standard form model as in~\eqref{eqn:PolesOfOrder1And2}.  Different choices of $a$ and $b$ yield isomorphic models.  By Corollary~\ref{cor:reconstruct} this implies that $\overline{\F}_3[a,b]^G=\overline{\F}_3[I_1,I_2,I_3]$. 

\end{proof}

\begin{remark} Given a more general model $y^3-y=\frac{ax+b}{x-\lambda_1}+\frac{cx^2+dx+e}{(x-\lambda_2)^2}$, the set
$$\left\{\frac{(c\lambda_1\lambda_2-d(\lambda_1+\lambda_2)+e)^4}{(c\lambda_2^2+d\lambda_2+e)^2}, \frac{(a\lambda_1+b)^4(c\lambda_2^2+d\lambda_2+e)^2}{(\lambda_1-\lambda_2)^8}, \frac{(c\lambda_1\lambda_2-d(\lambda_1+\lambda_2)+e)(a\lambda_1+b)}{(\lambda_1-\lambda_2)^2}\right\}$$
is a reconstructing set of invariants.
\end{remark}

\subsection{Genus 3, characteristic 7}
There is only one component in this stratum since the only partition is $\{3\}$.  In this case, there is only one pole of order 2 and a unique standard form:
$$y^7-y=x^2.$$
Note that this is one of the exceptional curves in case~\eqref{case:exceptionalCurve2} of Corollary~\ref{cor:generalASisomorphism}.

\subsection{Genus 4, characteristic 3}
This case has three strata, consisting of curves with respective $3$-ranks 0, 2, and 4.

\subsubsection{A single pole of order 5} \label{SS:one-pole-order-5}
In this case the $3$-rank is $s=0$ by~\eqref{eq:ASprank}, and we have $\vec{E}=\{6\}$. This corresponds to a curve $C_f$ where $f(x)$ has one pole of order 5. The standard form as given in Theorem~\ref{T:normal} is:
\begin{equation}\label{eqn:p=3,g=4,E=6}
C: y^3-y=x^5+cx^4+dx^2   
\end{equation}
for some $c,d\in \overline{\F}_5$. Note that none of the exceptional curves have a standard form of this type, so Corollary~\ref{cor:generalASisomorphism} applies to this stratum.

Identifying isomorphisms between standard forms in this case is cumbersome.  However, it is fairly straightforward to compute invariants from a different model of the curve.  We first apply isomorphisms to reach a different distinguished form, compute invariants using this model, and then express these invariants in terms of the coefficients in our standard form.  

\begin{lemma}\label{lem:reduced_form}
    Every curve with standard form $C$ given in~\eqref{eqn:p=3,g=4,E=6} is isomorphic to one of the form
\begin{equation}\label{eqn:AltForm1}
C_{g}: y^3-y=x^5+ax^2+bx   
\end{equation}
for some $a,b\in\overline{\F}_3$.  Specifically, $a=c^3+d$ and $b=(-cd-\epsilon^2)$, where $\epsilon^3=c$.
\end{lemma}
\begin{proof}
    We apply the isomorphism $\varphi_{\lambda,M}$ from~\eqref{eq:standardASisom} to $C$, where $\lambda=1$ and 
    $M=\begin{pmatrix} 1 & c \\ 0 & 1  \end{pmatrix}$, to obtain
    $$y^3-y=x^5-c^2x^3+(c^3+d)x^2-cdx+(-c^5+c^2d)=:\tilde{f}(x).$$ Next, we determine a polynomial $h(x)=h_1x+h_2 \in\overline{\F}_3[x]$ such that $h(x)^3-h(x)$ has the same $x^3$-term and constant term as $\tilde{f}$.  Let $\epsilon\in\overline{\F}_3$ with $\epsilon^3=c$.  
    Comparing coefficients at $x^3$ yields 
    $h_1^3=-c^2$, so $h_1=-\epsilon^2$. Comparing constant coefficients defines $h_2$ via the relation $h_2^3-h_2=-c^5+c^2d$.  Then $C_f$ is isomorphic to 
    \begin{align*}C_g:y^3-y&= x^5-c^2x^3+(c^3+d)x^2-cdx+(-c^5+c^2d)-\left((h_1x+h_2)^3-(h_1x+h_2)\right)\\
    &=x^5+(c^3+d)x^2+(-cd-\epsilon^2)x.
    \end{align*}
\end{proof}

\begin{proposition}\label{prop:siglePoleOrder5}
    Every isomorphism between curves defined as in~\eqref{eqn:AltForm1} is given, up to composition with powers of $\sigma:\,(x,y)\mapsto(x,y+1)$, by \[(x,y)\mapsto(\alpha x,\alpha^5 y)\] for some $\alpha\in\overline{\F}_3$ with $\alpha^{10}=1$.
\end{proposition}

\begin{proof}
    Let $M\colonequals\begin{pmatrix}
\alpha & \beta \\
\gamma & \delta
\end{pmatrix}\in \GL_2(\overline{\F}_3)$. 
Any isomorphism $\varphi_{\lambda,M}$ between curves of the form~\eqref{eqn:AltForm1} must preserve the pole at infinity, implying $\gamma=0$.  Further,  the requirement of a vanishing $x^4$ term forces $\beta=0$, and thus $h(x)=0$. The fact that the resulting form must be monic in $x$ necessitates $\lambda^{-1}\alpha^5/\delta^5=1$.  Without loss of generality, assume that $\delta=1$, so $\lambda=\alpha^5$.  Since $\lambda^2=1$, we must have $\alpha^{10}=1$.
\end{proof}

These isomorphisms yield a group action of the finite group $G\simeq \Z/5\Z$ acting linearly on $\overline{\F}_3[a,b]$, generated by $(a,b)\mapsto(a/\alpha^3, b/\alpha^4)$ with $\alpha^{10}=1$.

\begin{corollary}\label{cor:reduced_form_invars1}
    The elements $I_1= a^{10}$, $I_2= b^5$, and $I_3= a^2b$ generate the ring of invariants for Artin-Schreier curves of genus $4$ in characteristic $3$ with $3$-rank $0$, 
i.e.\ $\overline{\F}_3[a,b]^G=\overline{\F}_3[I_1,I_2,I_3]$. These elements satisfy the relation $I_1I_2=I_3^5$.
\end{corollary}

\begin{proof}
For a curve $C_g$ as in~\eqref{eqn:AltForm1}, we apply an isomorphism as given in Proposition~\ref{prop:siglePoleOrder5} to obtain
\[C_{\tilde{g}}:y^3-y=x^5+\frac{a}{\alpha^3}x^2+\frac{b}{\alpha^4}x.\]
It is easy to check that the given functions $I_1, I_2, I_3$. are invariant under this action.  Conversely, given values $\{I_1,I_2,I_3\}$ with $I_1I_2=I_3^5$, we can choose $a,b\in\overline{\F}_3$ such that $a^{10}=I_1,$ $b^5=I_2$, in which case $a^2b=I_3$. Suppose $I_1 = \tilde{a}^{10}$ and $I_2 = \tilde{b}^5$ with $\tilde{a},\tilde{b}\in\overline{\F}_3$, and let $\alpha\in\overline{\F}_3$ be such that $\alpha^3=\tilde{a}/a$. The relation $a^{10}=\tilde{a}^{10}$ implies that $\alpha^{10}=1$, so $\alpha^4=\tilde{a}^2/a^2$. Since $a^2b=\tilde{a}^2\tilde{b}$, we have $b/\alpha^4=\tilde{b}$. Hence, the curves $y^3-y=x^5+ax^2+bx$ and $y^3-y=x^5+\tilde{a}x^2+\tilde{b}x$ are isomorphic via the isomorphism in Proposition~\ref{prop:siglePoleOrder5} given by $\alpha$ with $\alpha^3=\tilde{a}/a$.
Corollary~\ref{cor:reconstruct} now implies that $I_1,I_2,$ and $I_3$ generate the invariant ring.
\end{proof}
By applying the change of variables from Lemma~\ref{lem:reduced_form}, we obtain the following.
\begin{corollary} Two curves in standard form as in~\eqref{eqn:p=3,g=4,E=6} are isomorphic if and only if they have the same invariants $\{I_1=(c^3+d)^{10}, I_2=(-cd-\epsilon^2)^5, I_3=(c^3+d)^2(-cd-\epsilon^2)\}$, where $\epsilon^3=c$.
\end{corollary}

\subsubsection{Two poles of order 2} This case corresponds to the partition $\vec{E}=\{3,3\}$.  According to Theorem~\ref{T:normal}, the curves in this stratum have standard form
\begin{equation}\label{eqn:p=3,g=4,E=3,3}
C\colon y^3-y=x^2+ax+\frac{b}{x}+\frac{c}{x^2}
\end{equation}
for $a,b,c\in\overline{\F}_3$ with $c\neq 0$.

\begin{proposition}\label{prop:g4p3poles22}
 Every isomorphism between curves in standard form as in~\eqref{eqn:p=3,g=4,E=3,3} is given, up to composition with powers of $\sigma:\,(x,y)\mapsto(x,y+1)$, by
\begin{equation}\label{eq:isomorphiamp3g3p2}
    (x,y)\mapsto (\alpha x, \alpha^3 y),
\end{equation}
where $\alpha^4=1$, or by
\begin{equation}\label{eq:isomorphiamp3g3b2}
    (x,y)\mapsto \left( \frac{1}{\gamma x}, \lambda y\right),
\end{equation}
where $\lambda\in \F_3^{\times}$ and $c\gamma^2=\lambda$.
\end{proposition}
\begin{proof} Let $C$ be as in equation~\eqref{eqn:p=3,g=4,E=3,3}.  Any isomorphism of $C$ that yields another curve in standard form must either fix $P_{\infty}$ and $P_0$ or swap these poles.  We first consider the case when both poles are fixed.  In this situation, once again disregarding composition by powers of $\sigma$, the isomorphism is of the form $\varphi_{\lambda,M_{\alpha}}$ for  $M_{\alpha}=\begin{pmatrix}
    \alpha & 0\\ 0 & 1
\end{pmatrix},$ with $\lambda\in\F_3^{\times}$ and $\alpha\in\overline{\F}_3^{\times}$.  Applying $\varphi_{M_{\alpha},\lambda}$ to $C$ gives
$$y^3-y=\frac{\alpha^2x^2}{\lambda}+\frac{a\alpha x}{\lambda}+\frac{b}{\lambda\alpha x}+\frac{c}{\lambda\alpha^2 x^2}.$$  Since the image curve must be in standard form, we have $\alpha^2=\lambda,$ so $\alpha^4=1$. This isomorphism acts on the standard form model 
by $(a,b,c)\mapsto\left(\frac{1}{\alpha}a,\alpha b,c\right).$

In the second case, when the poles are switched, the isomorphism must be of the form $\varphi_{\lambda,M_{\gamma}}$ for  $M_{\gamma}=\begin{pmatrix}
    0 & 1\\ \gamma & 0
\end{pmatrix},$ with $\lambda\in\F_3^{\times}$ and $\gamma\in\overline{\F}_3^{\times}$. The image of $C$ under $\varphi_{M_{\gamma},\lambda}$ is  
$$y^3-y=\frac{c\gamma^2x^2}{\lambda}+\frac{b\gamma x}{\lambda}+\frac{a}{\lambda\gamma x}+\frac{1}{\lambda\gamma^2 x^2}$$ 
so $c\gamma^2=\lambda$. Since $\lambda^2=1$, the isomorphism acts on the coefficients of a curve in standard form given by equation \eqref{eqn:p=3,g=4,E=3,3} via $(a,b,c)\mapsto\left(\frac{\gamma b}{\lambda},\frac{\lambda a}{\gamma},c\right)$. 
\end{proof}

The isomorphisms in Proposition~\ref{prop:g4p3poles22} define a non-linear group action of the dihedral group $G\simeq D_4$ of order 8 on the function field $\overline{\F}_3(a,b,c)$. It cannot a priori be descended to the polynomial ring $\overline{\F}_3[a,b,c]$. We proceed as follows to obtain invariants in this case.
With the notation in Equation (\ref{eq:isomorphiamp3g3b2}) in Proposition~\ref{prop:g4p3poles22}, we have $\left(\frac{\lambda}{\gamma}\right)^2=\lambda c=\pm c$. This defines a linear action of the dihedral group $ D_4$ of order 8  on the polynomial ring $\overline{\F}_3(\sqrt{c})[a,b]$ over the field $\overline{\F}_3(\sqrt{c})$. It is generated by $(a,b)\mapsto(ia,-ib)$ where $i^2=-1$ and $(a,b)\mapsto\left(\frac{b}{\sqrt{c}},\sqrt{c}a\right)$.  

\begin{lemma} The elements $I_2= ab$ and $I_3= a^4c^2-b^4$ generate $\overline{\F}_3(\sqrt{c})[a,b]^{G}$.
\end{lemma}

\begin{corollary} Two curves in standard form as in equation~\eqref{eqn:p=3,g=4,E=3,3} are isomorphic if and only if they have the same invariants $\{I_1= c,I_2= ab,I_3= a^4c^2-b^4\}$.
\end{corollary}

\subsubsection{Three poles of order 1} 
We now consider the stratum $\mathcal{AS}_{4,4}$ for $p=3$.  There is a single partition yielding $g=4$ and $s=4$, namely $\vec{E}=\{2,2,2\}$.  From Table~\ref{tab:smallgenustable}, we see that this component has dimension $\dim \mathcal{AS}_{4,\{2,2,2\}}=3$.  From Theorem~\ref{T:normal}, the standard form of a curve in this component is \begin{equation}\label{eqn:p=3,g=4,E=2,2,2}
    C \colon y^3-y=ax+\frac{b}{x}+\frac{c}{x-1},
\end{equation} where $a,b,c\in\overline{\F}_3^{\times}$. 

Note that there is a subfamily of exceptional curves (see case~\eqref{case:exceptionalCurve1} of Corollary~\ref{cor:generalASisomorphism}) within this stratum, of the form $C_{a}=y^3-y=\frac{a}{x^3-x}$ for $a\in\overline{F}_3^{\times}$. A standard form for this curve is given by $y^3-y=-ax+\frac{a}{x}+\frac{a}{x-1}$. 

\begin{proposition} Let $C$ be as in~\eqref{eqn:p=3,g=4,E=2,2,2}. The curves in standard form that are isomorphic to $C$ are the ones in Table~\ref{tab:p3g4E222}.
\end{proposition}

\begin{proof}
Due to Corollary~\ref{cor:generalASisomorphism} and Lemma~\ref{lem:standard1}, isomorphisms, up to composition with powers of $\sigma:\,(x,y)\mapsto(x,y+1)$, of these standard forms must be of the form 
\[\varphi_{\lambda,M}:(x,y)\mapsto (Mx,\lambda y +h_M(x)),\] where $\lambda\in\F_3^{\times}=\{\pm1\}$ and $M\in\GL_2(\overline{\F}_{3})$ permutes the poles $P_{\infty}$, $P_0$, and $P_1$ in an $S_3$ action.  This gives 6 choices of $M$ and 2 choices of $\lambda$, yielding at most 12 labeled models.  These 12 models are all distinct, and the action of the isomporphisms on $C$ is described in Table~\ref{tab:p3g4E222}.  We compute a sample entry to demonstrate the process.  To exchange $P_{\infty}$ and $P_0$ while fixing $P_1$, we solve using the reasoning in Section~\ref{subsec: isomorphisms} to find $M=\begin{pmatrix} 0 & 1\\1&0 \end{pmatrix}$.  The effect of $M$ on $f(x)=ax+\frac{b}{x}+\frac{c}{x-1}$ is
\[M(f(x))=\frac{a}{x}+bx+\frac{c}{\frac{1}{x}-1}%
= bx+\frac{a}{x}+\frac{-c}{x-1}-c.\]
We convert the corresponding curve equation to standard form, which is accomplished by sending $y\mapsto y+m$, where $m\in\overline{\F}_3$ such that $m^3-m=-c$.  This yields the standard form \[\varphi_{\lambda,M}(C): y^3-y=bx+\frac{a}{x}+\frac{-c}{x-1}. \qedhere\]
\end{proof}

\begin{table}[ht]
\bgroup
\def\arraystretch{2.2}
\begin{tabular}{l|l|l||c}
$\bm{\lambda}$ & \textbf{Permutation}                      & \textbf{Matrix} $\bm{M}$                                                                                   & $\bm{f(x)}$ \\ \hline\hline
1         & $(P_{\infty})$ & \bgroup
\def\arraystretch{1}$\begin{pmatrix} 1 & 0\\0 & 1\end{pmatrix}$\egroup        & $ax+\frac{b}{x}+\frac{c}{x-1}$   \\

1         & $(P_{\infty} P_0)$               & \bgroup
\def\arraystretch{1}$\begin{pmatrix} 0 & 1\\ 1 & 0\\ \end{pmatrix}$ \egroup  & $bx+\frac{a}{x}-\frac{c}{x-1}$\\
1         & $(P_{\infty} P_1 P_0)$           & \bgroup
\def\arraystretch{1}$\begin{pmatrix} 0 & -1\\ 1 & -1\\ \end{pmatrix}$ \egroup& $-bx+\frac{c}{x}-\frac{a}{x-1}$   \\
1         & $(P_{\infty} P_0 P_1)$           & \bgroup
\def\arraystretch{1}$\begin{pmatrix} -1 & 1\\ -1 & 0\\ \end{pmatrix}$\egroup& $-cx-\frac{a}{x}+\frac{b}{x-1}$   \\
1         & $(P_0 P_1)$               & \bgroup
\def\arraystretch{1}$\begin{pmatrix} 1 & -1\\0 & -1\\ \end{pmatrix}$  \egroup & $-ax-\frac{c}{x}-\frac{b}{x-1}$  \\
1         & $(P_{\infty} P_1)$                      & \bgroup\def\arraystretch{1}$\begin{pmatrix} -1 & 0\\ -1 & 1\\ \end{pmatrix}$\egroup& $cx-\frac{b}{x}+\frac{a}{x-1}$  \\
-1        & $(P_{\infty})$ & \bgroup
\def\arraystretch{1}$\begin{pmatrix} 1 & 0\\0 & 1\end{pmatrix}$\egroup      & $-ax-\frac{b}{x}-\frac{c}{x-1}$  \\
-1        & $(P_{\infty} P_0)$               & \bgroup
\def\arraystretch{1}$\begin{pmatrix} 0 & 1\\ 1 & 0\\ \end{pmatrix}$\egroup  & $-bx-\frac{a}{x}+\frac{c}{x-1}$ \\
-1        & $(P_{\infty} P_1 P_0)$           & \bgroup
\def\arraystretch{1}$\begin{pmatrix} 0 & -1\\ 1 & -1\\ \end{pmatrix}$\egroup & $bx-\frac{c}{x}+\frac{a}{x-1}$ \\
-1        & $(P_{\infty} P_0 P_1)$           & \bgroup
\def\arraystretch{1}$\begin{pmatrix} -1 & 1\\ -1 & 0\\ \end{pmatrix}$\egroup & $cx+\frac{a}{x}-\frac{b}{x-1}$ \\
-1        & $(P_0 P_1)$               & \bgroup
\def\arraystretch{1}$\begin{pmatrix} 1 & -1\\0 & -1\\ \end{pmatrix}$\egroup  & $ax+\frac{c}{x}+\frac{b}{x-1}$ \\
-1        & $(P_{\infty} P_1)$                      & \bgroup
\def\arraystretch{1}$\begin{pmatrix} -1 & 0\\ -1 & 1\\ \end{pmatrix}$\egroup& $-cx+\frac{b}{x}-\frac{a}{x-1}$ \medskip

\end{tabular}
\egroup
\caption{Action of the isomorphism $\varphi_{\lambda, M}$ on the coefficients of \eqref{eqn:p=3,g=4,E=2,2,2}.}
\label{tab:p3g4E222}
\end{table}

These isomorphisms define a linear action of a group $G\simeq D_6$, the dihedral group of 12 elements, generated by $(a,b,c)\mapsto(c,a,-b)$ and $(a,b,c)\mapsto(b,a,-c)$ on the polynomial ring $\overline{\F}_3[a,b,c]$.

\begin{corollary}
    The elements $I_1= (abc)^2$, $I_2= (abc)(a-b-c)$, $I_3= ab+ac-bc$, and $I_4=a^2+b^2+c^2$ generate the ring of invariants for Artin-Schreier curves of genus $4$ in characteristic $3$ with $3$-rank $4$, i.e.\ $\overline{\F}_3[a,b,c]^G=\overline{\F}_3[I_1,I_2,I_3,I_4]$. These elements satisfy the relation $I_1(I_3+I_4)=I_2^2$.
\end{corollary}

\begin{proof}
A model as in~\eqref{eqn:p=3,g=4,E=2,2,2} can be reconstructed from this set of invariants through a naive solve-and-back-substitution procedure that results in a polynomial equation of degree 6. Since all the roots of this equation belong to an (at most degree 6) extension of the (finite) field of definition of the invariants, there is a solution that defines coefficients for~\eqref{eqn:p=3,g=4,E=2,2,2}.
\end{proof}

\begin{remark}
The exceptional curves in this family occur exactly when $I_2=I_3=I_4=0$. 
\end{remark}

\subsection{Genus 4, characteristic 5} In this case there are two strata.

\subsubsection{One pole of order 3}
Here, the $5$-rank is $s=0$ and we have $\vec{E}=\{4\}$.  This corresponds to a curve $C_f$ where $f$ has one pole of order 3, with standard form as given in case~\eqref{case:r=0} of Theorem~\ref{T:normal}:
\begin{equation}\label{eqn:p=5,g=4,E=4}
C:y^5-y=x^3+ax^2
\end{equation}
with $a\in\overline{\F}_5$.  When $a=0$, the curve has additional automorphisms as in Lemma~\ref{lemma:explicitextraaut}.

\begin{proposition}\label{prop:isom_p=5,g=4,E=4}
    Every isomorphism between curves in standard form as in~\eqref{eqn:p=5,g=4,E=4} is given, up to composition with powers of $\sigma:\,(x,y)\mapsto(x,y+1)$, by $\varphi_{\lambda,M}$ with $\lambda=\alpha^3$ and $M=\begin{pmatrix}
        \alpha & \beta\\
        0 & 1
    \end{pmatrix}$, where $\alpha\in\overline{\F}_5$ with $\alpha^{12}=1$ and $\beta\in\left\{0,-\frac{2a}{3}\right\}$.
\end{proposition}

\begin{proof}
    Let $M\colonequals\begin{pmatrix}
\alpha & \beta \\
\gamma & \delta \\
\end{pmatrix}\in \GL_2(\overline{\F}_5)$. Any isomorphism $\varphi_{\lambda,M}$ between standard forms given by~\eqref{eqn:p=5,g=4,E=4} must preserve the pole at infinity, so $\gamma=0$.  Without loss of generality, assume $\delta=1$. We then obtain the image curve with model
$$y^5-y=\frac{1}{\lambda}\alpha^3x^3+(3\alpha^2\beta+a\alpha^2)x^2+(3\alpha\beta^2+2a\alpha\beta)x+a\beta^2.$$
Since the right hand side must be monic, we have $\alpha^3=\lambda$. Since $\lambda\in\F_5^{\times}$, $\lambda^4=1$, so $\alpha^{12}=1$.  Furthermore, since the standard form requires a vanishing $x$-term, it follows that $3\alpha\beta^2+2a\alpha\beta=0$, so $\beta=-\frac{2a}{3}$ or $\beta=0$. Finally, we choose an appropriate $h(x)=h_0\in\overline{\F}_5$ to eliminate the constant term $a\beta^2$.  
\end{proof}

We obtain a linear action of a group $G\simeq \Z/12\Z$ generated by $a\mapsto \alpha a$ with $\alpha^{12}=1$ on the polynomial ring $\overline{\F}_5[a]$.

\begin{corollary}
    The element $I_1= a^{12}$ generates the ring of invariants for Artin-Schreier curves of genus $4$ in characteristic $5$ with $5$-rank $0$, 
i.e.\ $\overline{\F}_5[a]^G=\overline{\F}_5[I_1]$.
\end{corollary}
\begin{proof} It is straightforward to check that  $\{I_1\}$ is invariant and a reconstructing set. Now apply Corollary~\ref{cor:reconstruct}.
\end{proof}

\subsubsection{Two poles of order 1}
In this case the $5$-rank is $s=4$ and we have $\vec{E}=\{2,2\}$.  This corresponds to a curve $C_f$ where $f$ has two poles of order 1, with standard form as given in case~\eqref{case:r=1} of Theorem~\ref{T:normal}:
\begin{equation}\label{eqn:p=5,g=4,E=2,2}
y^5-y=x+\frac{a}{x}
\end{equation}
with $a\in\overline{\F}_5^\times$.

\begin{proposition} Every isomorphism between curves in standard form as in \eqref{eqn:p=5,g=4,E=2,2} is given, up to composition with powers of $\sigma:\,(x,y)\mapsto(x,y+1)$, by 
\begin{equation}\label{eq:isomorphiamp3g3s4}
    (x,y)\mapsto (\lambda x,\lambda y)
\end{equation}
or by
\begin{equation}\label{eq:isomorphiamp3g3s4b}
    (x,y)\mapsto \left(\frac{a}{\lambda x} ,\lambda y\right),
\end{equation}
where $\lambda\in\F_5^{\times}$.
\end{proposition}

\begin{proof}
    Let $M\colonequals\begin{pmatrix}
\alpha & \beta \\
\gamma & \delta \\
 
\end{pmatrix}\in \GL_2(\overline{\F}_5)$ and $\lambda\in\F_5^{\times}$. Any isomorphism $\varphi_{\lambda,M}$ between standard forms here must either fix the poles $P_{\infty}$ and $P_0$ or swap them. So, up to scalar multiple, $M$ must be of the form $M_\alpha=\begin{pmatrix}
\alpha & 0 \\
0 & 1 \\
\end{pmatrix}$ or $M_\gamma=\begin{pmatrix}
0 & 1 \\
\gamma & 0 \\
\end{pmatrix}$. Let $f(x)=x+\frac{a}{x}$. In the first case, we obtain \[\varphi_{\lambda,M_{\alpha}}(C_f):y^5-y=\frac{\alpha}{\lambda}x+\frac{a}{\alpha\lambda x}.\]  
Since the polynomial part on the right hand side must be monic, we must have $\alpha=\lambda$, leading to the standard form $y^5-y=x+\frac{a}{\lambda^2 x}$ for  $\varphi_{\lambda,M_{\alpha}}(C_f)$. So the isomorphism acts on the standard form model by sending $a\mapsto \frac{a}{\lambda^2}=\lambda^2a$.  In the second case, the image curve is \[\varphi_{\lambda,M_{\gamma}}(C_f):y^5-y=\frac{1}{\lambda\gamma x}+\frac{a\gamma x}{\lambda}.\]  Again, since the polynomial part must be monic in standard form, we have $a\gamma=\lambda$, so the standard form of $\varphi_{\lambda,M_{\gamma}}(C_f)$ can be simplified to $y^5-y=x+\frac{a}{\lambda^2 x}$.  This isomorphism again acts on the standard form model by sending $a\mapsto \frac{a}{\lambda^2}=\lambda^2a$. 
\end{proof}

We obtain a linear action of a group $G\simeq \Z/2\Z$ generated by $a\mapsto -a$ on the polynomial ring $\overline{\F}_5[a]$.

\begin{corollary} \label{prop:twoPolesOfOrder1}
    The element $I_1= a^2$ generates the ring of invariants for Artin-Schreier curves of genus $4$ in characteristic $5$ with $5$-rank $4$, 
i.e.\ $\overline{\F}_5[a]^G=\overline{\F}_5[I_1]$.
\end{corollary}

\begin{proof} The set $\{I_1=a^2\}$ is a reconstructing set for curves in the stratum with standard form as in~\eqref{eqn:p=5,g=4,E=2,2}.
\end{proof}

\begin{remark} The curves in this family are all hyperelliptic, but of characteristic different from~$2$. They can be written as $v^2=u^{10}-2u^6+u^2-4a$ with $v=2x-y^5+y$ and $u=y$. Invariants for them inside the family of hyperelliptic curves of genus 4 in characteristic 5 (so inside a larger family) could also be computed with the techniques of \cite{LR12}. 
\end{remark}

\section{Conclusion}
In this paper, we determined reconstructing invariants for Artin-Schreier curves of genus~3 and~4 in characteristic $p>2$.  
Our results on genus 3 represent an important step toward the full characterization of invariants of curves of that genus.

Moving to higher genus curves presents many additional challenges.  In particular, curves~$C_f$ where $f$ has many poles of the same order exhibit many ismorphisms between standard forms.  As the genus grows, the problem becomes more and more complicated. Instead of seeking a general closed form for invariants of all Artin-Schreier curves, an algorithmic method of determining invariants seems appropriate.  Devising the right framework and specializing the results of invariant theory to this setting is the subject of ongoing work.

\bibliography{synthbib.bib}

\bibliographystyle{alpha}
\end{document}

%% file: macros.tex
\usepackage{amsmath,amssymb,dsfont,mathrsfs}
\usepackage{amsthm}
\usepackage{url}

\usepackage{pgf,tikz,pgfplots}
\pgfplotsset{compat=1.13}
\usetikzlibrary{arrows}
\usepackage{multicol}
\usepackage{bm}
\usepackage[all,cmtip]{xy}
\usepackage{tcolorbox}
\usepackage{tikz-cd}

\usepackage{mathtools}
\usepackage{colonequals}
\usepackage[letterpaper,top=1in,bottom=1in,left=1in,right=1in]{geometry}

\usepackage{enumitem}
\usepackage{latexsym}

\newenvironment{enumalphii}
{\begin{enumerate}}
{\end{enumerate}}

\usepackage[pagebackref=true]{hyperref}
\hypersetup{colorlinks=true,urlcolor=magenta,citecolor=magenta,linkcolor=magenta}

\renewcommand*{\backref}[1]{}
\renewcommand*{\backrefalt}[4]{%
  \ifcase #1 %
    \relax
  \or
    $\uparrow$#2.%
  \else
    $\uparrow$#2.%
  \fi%
}

\newcommand{\mathbbords}{\mathbb}

\newcommand{\Z}{{\mathbbords{Z}}}

\newcommand{\F}{\mathbbords{F}}

\newcommand{\Fpbar}{\overline{\mathbb{F}}_p}

\DeclareMathOperator{\GL}{GL}

\renewcommand{\P}{\ensuremath{\mathbbords{P}}}

\newcommand*{\MAGMA}{\textsf{Magma}}
\newcommand{\defi}[1]{\emph{#1}}

\usepackage{color}

\usepackage{float}

\theoremstyle{plain}\newtheorem{ithm}{Theorem}

\counterwithin{equation}{section}

\theoremstyle{plain}

\newtheorem{theorem}{Theorem}[section]
\newtheorem{proposition}[theorem]{Proposition}

\newtheorem{lemma}[theorem]{Lemma}
\newtheorem{corollary}[theorem]{Corollary}

\theoremstyle{definition}
\newtheorem{definition}[theorem]{Definition}

\theoremstyle{remark}
\newtheorem{remark}[theorem]{Remark}